\documentclass{amsproc}
\usepackage{amssymb,amsmath,amsthm}
\usepackage{amscd}
\input amssym.def
\input amssym

\newtheorem{remark}{Remark}[section]
\newtheorem{defi}{Definition}[section]
\newtheorem{prop}{Proposition}[section]

\newtheorem{notation}{Notation}[section]

\newcommand{\be}{\begin{equation}}
\newcommand{\ee}{\end{equation}}

\renewcommand{\theequation}{\thesection.\arabic{equation}}
\renewcommand{\thetheorem}{\thesection.\arabic{theorem}}

\renewcommand{\theequation}{\thesection.\arabic{equation}}

\numberwithin{equation}{section}

\subjclass[2000]{17B69,05A40}

\begin{document}

\title[] 
{New perspectives on exponentiated derivations, the formal
Taylor theorem, and Fa\`a di Bruno's formula}

\author{Thomas J. Robinson}


\begin{abstract}
We discuss certain aspects of the formal calculus used to describe
vertex algebras.  In the standard literature on formal calculus, the
expression $(x+y)^{n}$, where $n$ is not necessarily a nonnegative
integer, is defined as the formal Taylor series given by the binomial
series in nonnegative powers of the second-listed variable (namely,
$y$).  We present a viewpoint that for some purposes of generalization
of the formal calculus including and beyond ``logarithmic formal
calculus,'' it seems useful, using the formal Taylor theorem as a
guide, to instead take as the definition of $(x+y)^{n}$ the formal
series which is the result of acting on $x^{n}$ by a formal
translation operator, a certain exponentiated derivation.  These
differing approaches are equivalent, and in the standard generality of
formal calculus or logarithmic formal calculus there is no reason to
prefer one approach over the other.  However, using this second point
of view, we may more easily, and in fact do, consider extensions in
two directions, sometimes in conjunction.  The first extension is to
replace $x^{n}$ by more general objects such as the formal variable
$\log x$, which appears in the logarithmic formal calculus, and also,
more interestingly, by iterated-logarithm expressions.  The second
extension is to replace the formal translation operator by a more
general formal change of variable operator.  In addition, we note some
of the combinatorics underlying the formal calculus which we treat,
and we end by briefly mentioning a connection to Fa\`a di Bruno's
classical formula for the higher derivatives of a composite function
and the classical umbral calculus.  Many of these results are
extracted from more extensive papers \cite{R1} and \cite{R2}, to
appear.
\end{abstract}

\maketitle

\renewcommand{\theequation}{\thesection.\arabic{equation}}
\renewcommand{\thetheorem}{\thesection.\arabic{theorem}}

\section{Introduction}
Our subject is certain aspects of the formal calculus used, as
presented in \cite{FLM}, to describe vertex algebras, although we do
not treat any issues concerning ``expansions of zero,'' which is at
the heart of the subject.  An important basic result which we describe
in detail is the formal Taylor theorem and this along with some
variations is the topic we mostly consider.  It is well known, and we
recall the simple argument below, that if we let $x$ and $y$ be
independent formal variables, then the formal exponentiated derivation
$e^{y\frac{d}{dx}}$, defined by the expansion, $\sum_{k \geq
0}y^{k}\left(\frac{d}{dx}\right)^{k}/k!$, acts on a (complex)
polynomial $p(x)$ as a formal translation in $y$.  That is, we have
\begin{align}
e^{y\frac{d}{dx}}p(x)=p(x+y).\label{eq:PFTT}
\end{align}
Formulas of this type, where one shows how a formal exponentiated
derivation acts as a formal translation over some suitable space, such
as polynomials, are the content of the various versions of the formal
Taylor theorem.  In the standard literature on formal calculus, the
expression $(x+y)^{n}$ is defined as a formal Taylor series given by
the binomial series in nonnegative powers of the second-listed
variable.  This notational convention is called the ``binomial
expansion convention,'' as in \cite{FLM}; cf. \cite{LL}.  (Such series
expansions often display interesting underlying combinatorics, as we
discuss below.)  We note that there are really two issues in this
notational definition.  One is the relevant ``expansion'' of interest,
which is easy but substantial mathematically.  The other is purely a
``convention'', namely, deciding which listed variable should be
expanded in nonnegative powers.  Of course, one needs such a
definition before even stating a formal Taylor theorem since one needs
to know how to define what we mean when we have a formal function
whose argument is $(x+y)$.  The issue of how to define $\log (x+y)$
for use in the recently developed logarithmic formal calculus is
parallel.  This issue originally arose in \cite{M}, where the author
introduced logarithmic modules and logarithmic intertwining operators.
In that context it was necessary to handle nonnegative integral powers
of the logarithmic variables.  In fact, the definition given there was
\begin{align*}
\log(x+y)=e^{y\frac{d}{dx}} \log x,
\end{align*}
where $\log x$ is a formal variable such that $\frac{d}{dx} \log
x=1/x$ (see Section 1.3 and in particular Proposition 1.5 in
\cite{M}).  The logarithmic calculus was then further developed in
detail in Section 3 of \cite{HLZ}, where it was used in setting up
some necessary language to handle the recently developed theory of
braided tensor categories of non-semisimple modules for a vertex
algebra.  Actually, in \cite{HLZ} the authors proved a more general
formal Taylor theorem than they strictly needed, one involving general
complex powers.  We discuss this issue of the generality of exponents
below.  In \cite{HLZ}, the authors used a more standard approach
which, as we have been discussing, is to define the relevant
expressions $p(x+y)$ via formal analytic expansions and to then prove
the desired formal Taylor theorem.  We argue that, in fact, for
certain purposes it is more convenient to use formulas of the form
\eqref{eq:PFTT} as the definition of $p(x+y)$, as was done in \cite{M}
in the important special case mentioned above where $p(x+y)=\log
(x+y)$, whenever we extend beyond the elementary case of polynomials,
but most especially if one wishes to extend beyond the logarithmic
formal calculus.

Actually, the necessary structure is contained in the ``automorphism
property,'' which for polynomials $p(x)$ and $q(x)$ says that
\begin{align*}
e^{y\frac{d}{dx}}(p(x)q(x))
=\left(e^{y\frac{d}{dx}}p(x)\right)\left(e^{y\frac{d}{dx}}q(x)\right).
\end{align*}
The various formal Taylor theorems may then be interpreted as
representations of the automorphism property which specialize properly
in the easy polynomial case.  We note that from this point of view the
``expansion'' part of the binomial expansion convention is not a
definition but a consequence.  (The ``convention'' part, which tells
which listed variable should be expanded in the direction of
nonnegative powers is, of course, retained in both approaches as the
choice of notational convention.)

Whenever it was necessary to formulate more general formal Taylor
theorems, such as in \cite{HLZ}, it was heuristically obvious that
they could be properly formulated in the standard approach but as soon
as one generalizes beyond the case of the logarithmic calculus then
there may be some tedious details to work out.  It is hoped that the
approach presented here may in the future make such generalization
more efficient. In particular, we show how to generalize to a space
that involves formal logarithmic variables iterated an arbitrary
number of times as an example to show how this approach may be applied
to desired generalizations.

We noted that the traditional approach to proving generalized formal
Taylor theorems via formal analytic expansions may be tedious, and
while narrowly speaking this is true, it is also true that these
expansions are themselves interesting.  Indeed, once we have firmly
established the algebra of the automorphism property and the formal
Taylor theorem relevant to any given context we may calculate formal
analytic expansions.  If there is more than one way to perform this
calculation we may equate the coefficients of the multiple expansions
and find a combinatorial identity.  We record certain such identities,
which turn out to involve the well-known Stirling numbers of the first
kind and thereby recover and generalize an identity similarly
considered in Section 3 of \cite{HLZ}, which was part of the
motivation for this paper.

We are sometimes also interested in exponentiating derivations other
than simply $\frac{d}{dx}$.  For instance, in \cite{M} and \cite{HLZ}
the authors needed to consider the operator $e^{yx\frac{d}{dx}}$.
Such exponentiated derivations were considered in \cite{FLM}, and in
fact much more general derivations appearing in the exponent have been
treated at length in \cite{H}, but we shall only consider a couple of
very special cases like those mentioned already.  We present what we
call ``differential representations,'' which help us to transfer
formulas involving one derivation to parallel formulas for a second
one which can be interpreted as a differential representation of the
first.  The automorphism property holds true for all derivations, but
the formal Taylor theorem becomes a parallel statement telling us that
another formal exponentiated derivation acts as a formal change of
variable other than translation.  For example, for a polynomial
$p(x)$, one may easily show that
\begin{align*}
e^{yx\frac{d}{dx}}p(x)=p\left(xe^{y}\right).
\end{align*}

There is additional very interesting material which the automorphism
property, the formal Taylor theorem and the notion of differential
representation lead to.  For instance, it turns out that certain of
the basic structures of the classical umbral calculus, which was
studied by G.C. Rota, D. Kahaner, A. Odlyzko and S. Roman
(\cite{Rot2}, \cite{Rot1}, \cite{Rot3} and \cite{Rom}), and certain
aspects of the exponential Riordan group, which was studied by
L.W. Shapiro, S. Getu, W.-J. Woan and L. C. Woodson (\cite{Sh1} and
\cite{Sh2}), may be naturally formulated and recovered in a similar
context to the one we are considering.  In this paper we only indicate
this connection in a brief comment.  Such material is treated in
\cite{R2}.

In Section \ref{sec:tradform} we give an expository review of the
traditional formulation of formal Taylor theorems.  In Section
\ref{sec:nontradform} we reformulate the material of the previous
section from the point of view that formal Taylor theorems may be
regarded as representations of the automorphism property.  In Section
\ref{sec:varch} we consider a relation between the formal translation
operator and a second formal change of variable operator.  In Section
\ref{sec:comb} we record some underlying combinatorics recovering, in
particular, a classical identity involving Stirling numbers of the
first kind, which was rediscovered in \cite{HLZ}.  Finally, in Section
\ref{sec:umb} we briefly show a connection to Fa\`a di Bruno's
classical formula for the higher derivatives of a composite function
following a proof given in \cite{FLM}, as well as a related connection
to the umbral calculus.

Many of the results in this paper were presented at the Quantum
Mathematics and Algebra Seminars at Rutgers University and at the
International Conference on Vertex Operator Algebras and Related Areas
(a conference to mark the occasion of Geoffrey Mason's 60th birthday)
held at Illinois State University July 7-11 2008.  Many thanks for all
the helpful comments made by the members of those seminars and at the
conference, in particular comments made by Prof. S. Sahi,
Prof. R. Goodman and Prof. M. Bergvelt. Also, of course, many thanks
for all the helpful discussions with my advisor, Prof. J. Lepowsky.

\section{The formal Taylor theorem: a traditional approach}
\label{sec:tradform}
We begin by recalling some elementary aspects of formal calculus
(cf. e.g. \cite{FLM}).  We write $\mathbb{C}[x]$ for the algebra of
 polynomials in a single formal variable $x$ over the complex
numbers; we write $\mathbb{C}[[x]]$ for the algebra of formal power
series in one formal variable $x$ over the complex numbers, and we
also use obvious natural notational extensions such as writing
$\mathbb{C}[x][[y]]$ for the algebra of formal power series in one
formal variable $y$ over $\mathbb{C}[x]$.  Further, we shall
frequently use the notation $e^{w}$ to refer to the formal exponential
expansion, where $w$ is any formal object for which such expansion
makes sense.  For instance, we have the linear operator
$e^{y\frac{d}{dx}}:\mathbb{C}[x] \rightarrow \mathbb{C}[x][[y]]$:
$$e^{y\frac{d}{dx}}=\sum_{n \geq
0}\frac{y^{n}}{n!}\left(\frac{d}{dx}\right)^{n}.$$

\begin{prop} (The ``automorphism property'') Let $A$ be an algebra
over $\mathbb{C}$.  Let $D$ be a derivation on $A$.  That is,
$D$ is a linear map from $A$ to itself which satisfies the product
rule:
\begin{align*}
D(ab)=(Da)b+a(Db) \text{ for all } a\text{ and }b \text{ in }A.
\end{align*}
Then 
\begin{align*}
e^{yD}(ab)=\left(e^{yD}a\right)\left(e^{yD}b\right).
\end{align*}
\end{prop}

\begin{proof}
Notice that
$$
D^{n}ab=\sum_{n=0}^{r}\binom{r}{n}D^{r-n}aD^{n}b.
$$
Then divide both sides by $n!$ and sum over $y$ and the result follows.
\end{proof}

\begin{prop}
\label{prop:poltayl}
(The polynomial formal Taylor theorem) For $p(x) \in \mathbb{C}[x]$, we have

\begin{align*}
e^{y\frac{d}{dx}}p(x)=p(x+y).
\end{align*} 
\end{prop}

\begin{proof}
By linearity we need only check the case where $p(x)=x^{m}$, $m$ a
nonnegative integer.  We simply calculate as follows:

\begin{align*}
e^{y\frac{d}{dx}}x^{m}&=\sum_{n \geq
0}\frac{y^{n}}{n!}\left(\frac{d}{dx}\right)^{n}x^{m}\\ &=\sum_{n \geq
0}\frac{y^{n}}{n!}(m)(m-1)\cdots (m-(n-1))x^{m-n}\\ &=\sum_{n \geq
0}\binom{m}{n}x^{m-n}y^{n}\\ &=(x+y)^{m}.
\end{align*}
\end{proof}
Here, so far, we are, of course, using only the simplest,
combinatorially defined binomial coefficients, $\binom{m}{n}$ with
$m,n \geq 0$.  We observe that the only ``difficult'' point in the
proof is knowing how to expand $(x+y)^{m}$ as an element in
$\mathbb{C}[x][[y]]$.  In other words, the classical binomial theorem
is at the heart of the proof of the polynomial formal Taylor theorem
as well as at the heart of the proof of the automorphism property.

In order to extend the polynomial formal Taylor theorem to handle the
case of Laurent polynomials, we extend the binomial notation to
include expressions $\binom{m}{n}$ with $m <0$ and we also recall the
binomial expansion convention:

\begin{defi}
\label{binexpconv1}
 \rm We write
\begin{align}
\label{eq:binexpconv1}
(x+y)^{m}=\sum_{n\geq 0}\binom{m}{n}x^{m-n}y^{n} \quad, m \in \mathbb{Z},
\end{align}
where we assign to $\binom{m}{n}$ the algebraic (rather than
combinatorial) meaning: for all $m \in \mathbb{Z}$ and $n$ nonnegative integers
\begin{align*}
\binom{m}{n}=\frac{(m)(m-1)\cdots (m-(n-1))}{n!}.
\end{align*}
\end{defi}

\begin{remark} \rm
\label{rem:bece}
In the above version of the binomial expansion convention we may
obviously generalize to let $m \in \mathbb{C}$.
\end{remark}

With our extended notation, as the reader may easily check, the above
proof of Proposition \ref{prop:poltayl} exactly extends to give:
\begin{prop}(The Laurent polynomial formal Taylor theorem) For $p(x) \in
\mathbb{C}[x,x^{-1}]$, we have
\begin{align*}
e^{y\frac{d}{dx}}p(x)=p(x+y).
\end{align*} 
\end{prop}
\begin{flushright} $\square$ \end{flushright}

\begin{notation} \rm
We write $\mathbb{C}\{[x]\}$ for the algebra of finite sums of monomials of the
form $cx^{r}$ where $c$ and $r \in \mathbb{C}$.
\end{notation}

As the reader may easily check, the above proof of Proposition
\ref{prop:poltayl} exactly extends even further to give:
\begin{prop}
(The generalized Laurent polynomial formal Taylor theorem) For $p(x)
\in \mathbb{C}\{[x]\}$, we have
\begin{align*}
e^{y\frac{d}{dx}}p(x)=p(x+y).
\end{align*} 
\end{prop}

\begin{remark} 
\label{polyrem}
\rm There is an alternate approach to get the generalized Laurent
polynomial formal Taylor theorem, an approach which has the advantage
that no additional calculation is necessary in the final proof.  The
argument is simple.  For $r \in \mathbb{C}$, we need to verify that
\begin{align*}
e^{y\frac{d}{dx}}x^{r}=(x+y)^{r}.
\end{align*}
Now simply notice that both expressions lie in
\begin{align*}
\mathbb{C}x^{r}[x^{-1}][[y]]
\end{align*}
with coefficients being polynomials in $r$.  But the polynomials on
matching monomials agree for $r$ a nonnegative integer and so they must
be identical.  An argument in essentially this style appeared in
\cite{HLZ} to prove a logarithmic formal Taylor theorem (Theorem 3.6
of \cite{HLZ}).
\end{remark}

We now extend our considerations to a logarithmic case.

\begin{defi} \rm
Let $\log x$ be a formal variable commuting with $x$ and $y$ such that
$\frac{d}{dx} \log x=x^{-1}$.
\end{defi}

We shall need to define expressions involving $\log (x+y)$.  In
parallel with (\ref{eq:binexpconv1}) we shall define $(\log
(x+y))^{r}$, $r \in \mathbb{C}$, by its formal analytic expansion:

\begin{notation} \rm
We write
\begin{align}
\label{logdef}
(\log (x+y))^{r}&=\left( \log x + \log \left(1 +\frac{y}{x}\right)\right)^{r},
\end{align}
where we make a second use of the symbol ``$\log$'' to mean the usual
formal analytic expansion, namely
\begin{align*}
\log(1+X)=\sum_{i \geq 0}\frac{(-1)^{i-1}}{i}X^{i},
\end{align*}
and where we expand (\ref{logdef}) according to the binomial expansion
convention.
\end{notation}

\begin{remark} \rm
We note that (\ref{logdef}) is a special case of the definition used
in the treatment of logarithmic formal calculus in \cite{HLZ}.  Our
special case avoids the complication of the generality, treated in
\cite{HLZ}, of (uncountable, non-analytic) sums over $r \in
\mathbb{C}$.
\end{remark}

\begin{remark} \rm
The reader will need to distinguish from context which use of
``$\log$'' is meant.
\end{remark}

\begin{prop}
\label{prop:gplftt}
(The generalized polynomial logarithmic formal Taylor theorem) For
$p(x) \in \mathbb{C}\{ [x, \log x] \}$, we have
\begin{align*}
e^{y\frac{d}{dx}}p(x)=p(x+y).
\end{align*} 
\end{prop}

\begin{proof}
By linearity and the automorphism property, we need only check the
case $p(x)=(\log x)^{r},$ $ r \in \mathbb{C}$.  We could proceed by
explicitly calculating
\begin{align*}
e^{y\frac{d}{dx}}(\log x)^{r},
\end{align*}
but this is somewhat involved.  Instead we argue as in Remark
\ref{polyrem} to reduce to the case $r=1$.  Even without explicitly
calculating $e^{y\frac{d}{dx}}(\log x)^{r}$, it is not hard to see
that it is in
\begin{align*}
\mathbb{C}[r](\log x)^{r}\mathbb{C}[(\log x)^{-1},x^{-1}][[y]].
\end{align*}
When we expand (\ref{logdef}) we find that it is also in
\begin{align*}
\mathbb{C}[r](\log x)^{r}\mathbb{C}[(\log x)^{-1},x^{-1}][[y]].
\end{align*}
Thus we only need to check the case for $r$ a positive integer.  
A second application of the automorphism property now shows that we
only need the case where $r=1$.  This case is not difficult to
calculate:

\begin{align*}
e^{y\frac{d}{dx}} \log x&=\log x+\sum_{i \geq 1}
\frac{y^{i}}{i!}\left(\frac{d}{dx}\right)^{i}\log x \\
&=\log x+\sum_{i \geq 1}\frac{y^{i}}{i!}\left(\frac{d}{dx}\right)^{i-1}x^{-1}\\
&=\log x+\sum_{i \geq 1}\frac{y^{i}}{i}(-1)^{i-1}x^{-i}\\
&=\log x + \log\left(1 +\frac{y}{x}\right).
\end{align*}
\end{proof}

\begin{remark} \rm
Although we are working in a more special case than that considered in
\cite{HLZ}, the argument presented in the proof of Proposition
\ref{prop:gplftt} could be used as a replacement for much of the
algebraic proof of Theorem 3.6 in \cite{HLZ} as long as one is not
concerned with calculating explicit formal analytic expansions and
checking the corresponding combinatorics.  These two approaches are
very similar, however, the difference only being how much work is left
implicit.  In the next section we shall take a different point of view
altogether.
\end{remark}

\section{The formal Taylor theorem from a different point of view}
\label{sec:nontradform}
{}From the examples in Section \ref{sec:tradform} we see a common
strategy for formulating a formal Taylor theorem:

1) Pick some reasonable space (e.g., $\mathbb{C}[x]$,
   $\mathbb{C}\{[x,\log x]\}$) on which $\frac{d}{dx}$ acts in a
   natural way.  The space need not be an algebra, but in this paper
   we shall only consider this case.

2) Choose a plausible formal analytic expansion of relevant
   expressions involving $x+y$ (e.g., $(x+y)^{r},\, r \in \mathbb{C},
   \,\log(x+y)$).

3) Consider the equality $e^{y\frac{d}{dx}}p(x)=p(x+y)$ and either
   directly expand both sides to show equality or if necessary use a
   trick like in Remark \ref{polyrem}.

Step 2 is necessarily anticipatory and dependent on formal analytic
expressions.  Therefore it seems natural to replace Step 2 by simply
defining expressions involving $x+y$ in terms of the operator
$e^{y\frac{d}{dx}}$.  Then the formal Taylor theorem is trivially
true, being viewed now as a (plausible) representation of the
underlying structure of the automorphism property.  We redo the
previous work from this point of view.

\begin{prop} 
(The polynomial formal Taylor theorem) For $p(x) \in \mathbb{C}[x]$,
we have
\begin{align*}
e^{y\frac{d}{dx}}p(x)=p(x+y).
\end{align*} 
\end{prop}

\begin{proof}We have by the automorphism property:
$$
e^{y\frac{d}{dx}}p(x)=p\left(e^{y\frac{d}{dx}}x\right)=p(x+y).
$$
\end{proof}

Now for the replacement step:

\begin{defi} \rm
We write
\label{def:trans}
\begin{align*}
(x+y)^{r}=e^{y\frac{d}{dx}}x^{r} \qquad \text{for} \quad r \in \mathbb{C}.
\end{align*}
\end{defi}

\begin{remark} \rm
Of course, Definition \ref{def:trans} is equivalent to Definition
\ref{binexpconv1} together with Remark \ref{rem:bece}.  This definition
immediately leads to the most convenient proofs of certain
``expected'' basic properties, instead of needing to wait (as is often
done) to prove a formal Taylor theorem to officially obtain these
proofs.  For example, we have:

\begin{align*}
(x+y)^{r+s}&=e^{y\frac{d}{dx}}x^{r+s}\\
&=e^{y\frac{d}{dx}}(x^{r}x^{s})\\
&=\left(e^{y\frac{d}{dx}}x^{r}\right)\left(e^{y\frac{d}{dx}}x^{s}\right)\\
&=(x+y)^{r}(x+y)^{s}.
\end{align*}
\end{remark}

\begin{prop} (The generalized Laurent polynomial formal 
Taylor theorem) For $p(x) \in \mathbb{C} \{ [x] \}$, 
\begin{align*}
e^{y\frac{d}{dx}}p(x)=p(x+y).
\end{align*} 
\end{prop}

\begin{proof}
This is trivial.
\end{proof}

We also have this example of the replacement step:
\begin{defi} \rm
We write
\begin{align*}
(\log (x+y))^{r}=e^{y\frac{d}{dx}} (\log x)^{r} \qquad \text{for}
\quad r \in \mathbb{C}.
\end{align*}
\end{defi}

\begin{prop} 
(The generalized polynomial logarithmic formal Taylor theorem) For
$p(x) \in \mathbb{C} \{ [x, \log x] \}$,
\begin{align*}
e^{y\frac{d}{dx}}p(x)=p(x+y).
\end{align*} 
\end{prop}

\begin{proof}
The result follows by considering the trivial cases $p(x)=x^{r}$ and
$p(x)=(\log x )^{r}$ for $r \in \mathbb{C}$ and applying the
automorphism property.
\end{proof}

The formal analytic expansions are now viewed as calculations rather
than definitions or conventions.  So, for instance, we may calculate
the expansions (\ref{eq:binexpconv1}) and (\ref{logdef}) as
consequences rather than viewing them as definitions.

\section{More general formal changes of variable}
\label{sec:varch}
There are other formal Taylor-like theorems involving, for instance,
the exponentiated derivation, $e^{yx\frac{d}{dx}}$.
To recover such results we could repeat a complete parallel set of
reasoning beginning with the automorphism property applied to the
desired derivation.  However, instead of starting over from the
beginning, we show how to ``lift'' them from the formal Taylor
theorems we have already proved.  This sort of method has the added
benefit of showing relationships between different derivations instead
of obtaining isolated results.

To proceed properly we need to look at one more extension of the
formal Taylor theorem. To this end we let $\ell_{n}(x)$ be formal
commuting variables for $n \in \mathbb{Z}$. We define an action of
$\frac{d}{dx}$, a derivation, on
$$\mathbb{C}[\dots,\ell_{-1}(x)^{\pm 1},\ell_{0}(x)^{\pm
1},\ell_{1}(x)^{\pm 1},\dots]$$ (which for short we denote by
$\mathbb{C}[\ell^{\pm 1}]$) by
\begin{align*}
\frac{d}{dx}\ell_{-n}(x)&=\prod_{i=-1}^{-n}\ell_{i}(x),\\
\frac{d}{dx}\ell_{n}(x)&=\prod_{i=0}^{n-1}\ell_{i}(x)^{-1},\\
\text{\rm and   } \qquad \frac{d}{dx}\ell_{0}(x)&=1,
\end{align*}
for $n > 0$.  Secretly, $\ell_{n}(x)$ is the $(-n)$-th iterated
exponential for $n < 0$ and the n-th iterated logarithm for $n > 0$
and $\ell_{0}(x)$ is $x$ itself.  We make the following, by now typical,
definition in order to obtain a formal Taylor theorem.

\begin{defi} \rm
Let
\begin{align*}
\ell_{n}(x+y)=e^{y\frac{d}{dx}}\ell_{n}(x) \qquad \text{for} \quad n
\in \mathbb{Z}.
\end{align*}
\end{defi}
This gives:

\begin{prop}
(The iterated exponential/logarithmic formal Taylor theorem) For $p(x)
\in \mathbb{C}[\ell^{\pm 1}]$ we have:
\begin{align*}
e^{y\frac{d}{dx}}p(x)=p(x+y).
\end{align*} 
\end{prop}

\begin{proof}
The result follows from the automorphism property.
\end{proof}

Now consider the substitution map 
\begin{align*}
\phi: \mathbb{C}[\ell^{\pm 1}] \rightarrow \mathbb{C}[\ell^{\pm 1}]
\end{align*}
and its inverse defined by 
\begin{align*}
\phi (\ell_{n}(x))&=\ell_{n+1}(x) \qquad \text{for} \quad n \in \mathbb{Z}\\
( \text{\rm and } \qquad \phi^{-1} (\ell_{n}(x))&=\ell_{n-1}(x) \qquad
\text{for} \quad n \in \mathbb{Z}).
\end{align*}
 
\begin{prop}
We have
\begin{align*}
\phi \circ \frac{d}{dx}&=\ell_{0}(x)\frac{d}{dx} \circ \phi\\
\text{\rm and } \qquad \phi^{-1} \circ \ell_{0}(x)\frac{d}{dx}&=
\frac{d}{dx} \circ \phi^{-1}.
\end{align*}
\end{prop}

This proposition makes clear that, on the appropriate space,
$e^{y\frac{d}{dx}}$ and $e^{y\ell_{0}(x)\frac{d}{dx}}$ are simply shifted (in
terms of the subscripts of $\ell_{n}(x)$) versions of each other.

\begin{proof}
Since $\frac{d}{dx}$ and $\ell_{0}(x)\frac{d}{dx}$ are derivations we
need only check the action on $\ell_{n}(x)$ $n \in \mathbb{Z}$.  The
verification is routine calculation. For instance:

For $n > 1$
\begin{align*}
\ell_{0}(x)\frac{d}{dx}\phi \ell_{-n}(x)=\ell_{0}(x)\frac{d}{dx}
\ell_{-n+1}(x)&=\ell_{0}(x)\prod_{i=-1}^{-n+1}\ell_{i}(x)\\
&=\prod_{i=0}^{-n+1}\ell_{i}(x)\\
&=\phi \prod_{i=-1}^{-n}\ell_{i}(x)\\
&=\phi \frac{d}{dx}\ell_{-n}(x).
\end{align*}

\end{proof}

We then have the following two examples of the ``lifting'' process
referred to in the introduction to this section:

\begin{align*}
e^{y\ell_{0}(x)\frac{d}{dx}}\ell_{0}(x)&=\phi \circ e^{y\frac{d}{dx}}
\phi^{-1}(\ell_{0}(x)) \\ 
&=\phi \circ e^{y\frac{d}{dx}} \ell_{-1}(x) \\ 
&=\phi \circ \sum_{n \geq 0}\frac{y^{n}}{n!} \ell_{-1}(x) \\
&=\ell_{0}(x)e^{y},
\end{align*}
and
\begin{align*}
e^{y\ell_{0}(x)\frac{d}{dx}}\ell_{1}(x)
&=\phi \circ e^{y\frac{d}{dx}}
\phi^{-1}(\ell_{1}(x)) \\
 &=\phi \circ e^{y\frac{d}{dx}} \ell_{0}(x) \\ 
&=\phi(\ell_{0}(x)+y) \\ 
&=\ell_{1}(x)+y,
\end{align*}
which translate respectively to the following identities in more
standard logarithmic notation:
\begin{align*}
e^{yx\frac{d}{dx}}x&=xe^{y}\\
e^{yx\frac{d}{dx}}\log x&=\log x+y.
\end{align*}

\begin{remark} \rm
Of course, these examples can be obtained much more easily without
resorting to this method but in more involved examples this approach
is very useful (see e.g. \cite{R1}).
\end{remark}

\begin{remark} \rm
Although we do not give a precise definition here, it is maps like
$\phi$ that we call differential representations.  For more on these
differential representations see \cite{R1} and \cite{R2}.
\end{remark}

\section{Some combinatorics}
\label{sec:comb}
The original (algebraic) proof in \cite{HLZ} of the logarithmic formal Taylor
theorem used formal analytic expansions (in fact, so did the
statement).  We have bypassed those expansions in our approach, but they
are themselves of some interest.  For instance, the original proof relied
on a combinatorial identity arising from equating the coefficients of
two different formal analytic expansions.

We shall not get into the details of calculating formal analytic
expansions here, but instead, merely briefly state some results to
give the reader some idea of the material involved.  It is possible to
calculate the following three formal analytic expressions for
$\ell_{n}(x+y)^{r}$, where we fix $r \in \mathbb{C}$ (see \cite{R1}):

\begin{align*}
e^{y\frac{d}{dx}}\ell_{n}(x)^{r}=
\ell_{n}(x+y)^{r}&=
\sum_{j_{0}, \dots, j_{n} \geq 0}
\left( \prod_{i=0}^{n-1}
\begin{bmatrix}
j_{i} \\
j_{i+1} \\
\end{bmatrix}
\right) (-1)^{j_{0}+j_{n}} \frac{j_{n}!}{j_{0}!}\cdot\\ 
& \quad \cdot
\binom{r}{j_{n}} \ell_{n}(x)^{r}
\left(\prod_{i=0}^{n}\ell_{i}(x)^{-j_{i}}\right)y^{j_{0}}\\ 
&=\sum_{k
\geq 0}\frac{y^{k}}{k!}  \sum_{1 \leq j_{n} \leq \cdots \leq j_{1}
\leq j_{0}=k} j_{n}!\binom{r}{j_{n}}(-1)^{j_{0}+j_{n}} \cdot \\ 
&\quad \cdot S(j_{n},\dots, j_{0})
\ell_{n}(x)^{r-j_{n}}\ell_{n-1}^{-j_{n-1}} \cdots \ell_{0}^{-k}\\
&=\sum_{k \geq 0}\frac{y^{k}}{k!}\underset{0 \leq j_{0},j_{1},\dots,
j_{n}}{\sum_{j_{0}+j_{1}+\cdots+j_{n}=k}} j_{n}!\binom{r}{j_{n}}
\cdot\\ 
& \quad \cdot \left(\prod_{i=0}^{n-1}
(j_{i};\alpha_{i+1})\right)
\ell_{n}(x)^{r}\left(\prod_{i=0}^{n}\ell_{i}(x)^{-\alpha_{i}}\right),
\end{align*}
where
\begin{align*}
\alpha_{i}=\sum_{l=i}^{n}j_{l},
\end{align*}
\begin{align*}
\begin{bmatrix}
k \\
j \\
\end{bmatrix}
=\frac{k!}{j!}\underset{i_{l} \geq 1}{\sum_{i_{1}+\cdots+
i_{j}=k}}\frac{1}{i_{1}\cdots i_{j}},
\end{align*}
\begin{align*}
(m;n)=(-1)^{m}\sum_{0 \leq i_{1} < i_{2}< \cdots <i_{m} \leq
m+n-1}i_{1}i_{2}\cdots i_{m},
\end{align*}
and where $S(j_{n}, j_{n-1}, \dots , j_{0})$ is 
given by the following recursion:
\begin{align*}
S(j_{n}, \dots , j_{0})&=S(j_{n}-1,j_{n-1}-1, \dots , j_{0}-1)\\
& +(j_{n-1}-1)S(j_{n},j_{n-1}-1, \dots , j_{0}-1)\\
& \,\,\, \vdots\\
& +(j_{0}-1)S(j_{n},j_{n-1}, \dots ,j_{1}, j_{0}-1),
\end{align*}
along with the initial conditions,
\begin{align*}
S(j_{n},j_{n-1},\dots, j_{1}, 1)=
\left\{
\begin{array}{lr}
1&j_{n}=j_{n-1}=\cdots=j_{1}=1\\
0&\text{otherwise.}
\end{array}
\right.
\end{align*}
Equating coefficients yields the identity
\begin{align*}
S(j_{n},j_{n-1}, \dots ,j_{0})&=
\prod_{i=0}^{n-1}
\begin{bmatrix}
j_{i} \\
j_{i+1} \\
\end{bmatrix},
\end{align*}
for $j_{0} \geq j_{1} \geq \cdots \geq j_{n} \geq 1$ and $0 \leq s
\leq j_{n}$.  When we specialize to the $n=1$ case (the single
logarithm case), we get
\begin{align*}
(m;n)=(-1)^{m}\begin{bmatrix}
m+n \\
n \\
\end{bmatrix},
\end{align*}
and, more generally,
\begin{align}
\label{Lubell}
S(m,n)&=
\begin{bmatrix}
n \\
m \\
\end{bmatrix}=  
\frac{n!}{m!}\underset{i_{l} \geq 1}{\sum_{i_{1}+\cdots
+i_{m}=n}}\frac{1}{i_{1}\cdots i_{m}} \nonumber\\ 
& = \sum_{0 \leq
i_{1} <i_{2} < \cdots <i_{n-m} \leq n-1}i_{1}\cdots i_{n-m}, \\
\nonumber
\end{align}
where $S(m,n)$ satisfies a standard recurrence for the Stirling
numbers of the first kind.

\begin{remark} \rm
Most of the identity (\ref{Lubell}) appeared in Section 3 of
\cite{HLZ} in the course of a ``traditional-style'' algebraic proof of
a logarithmic formal Taylor theorem.  See also Remark 3.8 in
\cite{HLZ}, where this identity was observed to solve a problem posed
by D. Lubell in the Problems and Solutions section of the American
Mathematical Monthly \cite{Lu}.  The reappearance of this classical
identity in this context was one of the original motivations for the
present paper.
\end{remark}

\begin{remark} \rm
We also note that while the formal analytic expansions presented in
this section could serve as definitions they would obviously be
unwieldy.
\end{remark}

\section{Fa\`a di Bruno and umbral calculus}
\label{sec:umb}
There are some interesting variants of the notion of differential
representation.  In fact, one such variant appears implicitly in the
proof of Fa\`a di Bruno's formula in Proposition 8.3.4 of \cite{FLM},
an argument which is essentially the basis for proving the
(highly-nontrivial) ``associativity'' property of lattice vertex
operator algebras in a setting based on arbitrary rational lattices;
see Sections 8.3 and 8.4 of \cite{FLM}.  We present a special case of
this argument next.

Let $x,y,z$ be formal commuting variables.  Let $f(x), g(x) \in
\mathbb{C}[x]$. Then
\begin{align}
e^{y\frac{d}{dx}}f(g(x))&=f(g(x+y)) \nonumber\\
&=f(g(x)+(g(x+y)-g(x)))\nonumber\\ 
&=e^{(g(x+y)-g(x))\frac{d}{dz}}f(z)
|_{z=g(x)}\nonumber\\ 
&=\sum_{n \geq
0}\frac{f^{(n)}(z)(g(x+y)-g(x))^{n}}{n!} |_{z=g(x)}\nonumber\\
&=\sum_{n \geq 0}\frac{f^{(n)}(g(x))(g(x+y)-g(x))^{n}}{n!}\nonumber\\
&=\sum_{n \geq 0}\frac{f^{(n)}(g(x))
\left(e^{y\frac{d}{dx}}g(x)-g(x)\right)^{n}}{n!}\nonumber\\ 
&=\sum_{n\geq 0}\frac{f^{(n)}(g(x))\left(\sum_{m \geq
1}\frac{y^{m}g^{(m)}(x)}{m!}\right)^{n}}{n!}. \label{eq:FDBG}
\end{align}

Motivated by this, we consider the algebra
$\mathbb{C}[y_{0},y_{1},y_{2},\dots, x_{1},x_{2}, \dots]$ where
$y_{i}, x_{j}$ for $i \geq 0$ and $j \geq 1$ are commuting formal
variables.  Let $D$ be the unique derivation on
$\mathbb{C}[y_{0},y_{1},y_{2},\dots, x_{1},x_{2}, \dots]$ satisfying
the following:

\begin{align*}
Dy_{i}&=y_{i+1}x_{1} \qquad \qquad i \geq 0\\
Dx_{j}&=x_{j+1} \qquad \qquad \,\,\,\, j \geq 1.
\end{align*}
Then this question of calculating $e^{y\frac{d}{dx}}f(g(x))$ is seen
to be essentially equivalent to calculating
\begin{align*}
e^{zD}y_{0},
\end{align*}
where we ``secretly,'' loosely speaking, identify $\frac{d}{dx}$ with
$D$, $f^{(n)}(g(x))$ with $y_{n}$ and $g^{(m)}(x)$ with $x_{m}$ (and
$y$ with $z$).  The reader may note that we are now really dealing
with, among other things, a certain sort of completion of the original
problem, so that one may, for instance, wish to view $f(x)$ as a
formal power series and $g(x)$ as a formal power series with zero
constant term, and indeed we note that it was in this generality (and
with even more general derivations) that the above argument was
carried out in \cite{FLM}.  For a detailed description of this
material, we refer the reader to \cite{R2}.

Before proceeding, we note that we may write an intermediate step of
(\ref{eq:FDBG}) as
\begin{align*}
e^{y\frac{d}{dx}}\phi(f(z))=\phi\left(
e^{(g(x+y)-g(x))\frac{d}{dz}}f(z)
\right),
\end{align*}
where $\phi: \mathbb{C}[z] \rightarrow \mathbb{C}[x]$ substitutes
$g(x)$ for $z$.  That is, we have a commutative diagram:
\begin{align*}
\begin{CD}
\mathbb{C}[z] @>e^{(g(x+y)-g(x))\frac{d}{dz}}>> \mathbb{C}[x,y,z] \\
@VV\phi V                                          @VV\phi V \\
\mathbb{C}[x] @>e^{y\frac{d}{dx}}>> \mathbb{C}[x,y]. \\
\end{CD}
\end{align*}
This commutative diagram shows how $\phi$ may be regarded as a sort of
``global (i.e., exponentiated) differential representation.''  In our
new setting we might consider looking at ``nonglobal'' differential
representations of $D$.  This turns out to be too restrictive, but a
suitably loosened version of this question turns out to lead to
interesting results.

Let $\phi_{B}$ be the substitution which sends $y_{j}$ to $1$ for $j
\geq 0$ and sends $x_{i}$ to $xB_{i}$ for $i \geq 1$, where $B_{i} \in
\mathbb{C}$ for $i \geq 1$ is a fixed, arbitrary sequence subject to
the requirement that $B_{1} \neq 0$.

\begin{prop}
There is a unique linear map $D_{B}:\mathbb{C}[x] \rightarrow
\mathbb{C}[x]$ which satisfies the condition
\begin{align}
D_{B}^{n} \phi_{B} (y_{0})=\phi_{B} (D^{n}(y_{0})) \qquad n\geq
0. \label{eq:ushift}
\end{align}
\end{prop}

\begin{proof}
It is easy to see that $ \phi_{B} D^{n}(y_{0})$ is a polynomial of
degree $n$ whose leading term is $B_{1}^{n}x^{n}$, where we recall
that this coefficient is nonzero.  Thus each required equality in turn
(indexing by $n$) may be solved to obtain an equation of the form
$D_{B}x^{n-1}=r(x)$, where $r(x)$ is a polynomial of degree $n$.  Of
course this recursive process solves for and completely determines
$D_{B}x^{n}$ for all $n \geq 0$.
\end{proof}

The maps $D_{B}$ are what have been called umbral shifts, as in
\cite{Rom}.  For more on the connection to classical umbral calculus
see \cite{R2}.

\noindent {\small \sc Department of Mathematics, Rutgers University,
Piscataway, NJ 08854} 
\\ {\em E--mail
address}: thomasro@math.rutgers.edu

\begin{thebibliography}{Rot3}

\bibitem[FLM]{FLM}I. Frenkel, J. Lepowsky and A. Meurman, {\em Vertex
Operator Algebras and the Monster}, Pure and Appl. Math., Vol. 134,
Academic Press, New York, 1988.

\bibitem[H]{H} Y.-Z. Huang, {\em Two-dimensional Conformal Field
Theory and Vertex Operator Algebras}, Progress in Math., Vol. 148,
Birkh\"auser, Boston, 1998.

\bibitem[HLZ]{HLZ} Y.-Z. Huang, J. Lepowsky and L. Zhang,
Logarithmic tensor product theory for generalized modules for a
conformal vertex algebra, arXiv:0710.2687 [math.QA].

\bibitem[LL]{LL}
J. Lepowsky and H. Li, {\em Introduction to Vertex Operator Algebras
and Their Representations}, Progress in Math., Vol. 227, Birkh\"auser,
Boston, 2003.

\bibitem[Lu]{Lu} D. Lubell, Problem 10992, Problems and Solutions,
{\em American Mathematical Monthly} {\bf 110}, 2003, 155.

\bibitem[M]{M} A. Milas, Weak modules and logarithmic intertwining
operators for vertex operator algebras, in: {\em Recent Developments
in Infinite-Dimensional Lie Algebras and Conformal Field Theory},
ed. S. Berman, P. Fendley, Y.-Z. Huang, K. Misra, and B. Parshall,
Contemp. Math., Vol. 297, American Mathematical Society, Providence,
RI, 2002, 201--225.

\bibitem[R1]{R1} T.J. Robinson, The formal Taylor theorem revisited,
to appear.

\bibitem[R2]{R2} T.J. Robinson, Formal calculus and umbral calculus,
to appear.

\bibitem[Rom]{Rom} S. Roman, {\em The Umbral Calculus}, Pure and
Appl. Math., Vol. 111, Academic Press, New York, 1984.

\bibitem[Rot1]{Rot1} G.-C. Rota, {\em Finite Operator Calculus}, Academic
Press, New York, 1975.

\bibitem[Rot2]{Rot2} G.-C. Rota, D. Kahaner and A. Odlyzko, On the
foundations of combinatorial theory VIII: finite operator calculus, {
\em J. Math. Anal. Appl.} {\bf 42} (1973), 684--760.

\bibitem[Rot3]{Rot3} G.-C. Rota and S. Roman, The umbral calculus, {
\em Adv. in Math.} {\bf 27} (1978), 95--188.

\bibitem[Sh1]{Sh1} L.W. Shapiro, S. Getu, W.J. Woan and L.C.  Woodson,
The Riordan group, {\em Discrete Appl. Math.}  {\bf 34} (1991),
229--239.

\bibitem[Sh2]{Sh2} L. Shapiro, Bijections and the Riordan group, {\em
Theoret. Comput. Sci.} {\bf 307} (2003), 403--413.


\end{thebibliography}
\end{document}